\documentclass{amsart}
\usepackage{geometry, comment}                
\usepackage{amssymb}
\usepackage{hyperref}
\usepackage{pdfsync}

\newtheorem{thm}{Theorem}[section]
\newtheorem{lem}[thm]{Lemma}
\newtheorem{cor}[thm]{Corollary}

\theoremstyle{definition}

\newcommand{\BZ}{\mathbb{Z}}
\newcommand{\BN}{\mathbb{N}}

\newcommand{\factor}[2]{{\raise0.7ex\hbox{$#1$} \!\mathord{\left/ {\vphantom {#1 {#2}}}\right.\kern-\nulldelimiterspace}\!\lower0.7ex\hbox{${#2}$}}}

\DeclareMathOperator{\BS}{BS}

\title[Baumslag-Solitar Groups]{Two Remarks on First-Order Theories of Baumslag-Solitar Groups
\thanks{ Work done during a visit to the Omsk Branch of the Sobolev Mathematical Institute of the Siberian Branch of the Russian Academy of Science. The authors thank the staff of the Mathematical Institute for their hospitality and especially V. N. Remeslennikov.}}

\author[M. Casals-Ruiz]{Montserrat Casals-Ruiz}
\address{Department of Mathematics, 1326 Stevenson Center, Vanderbilt University, Nashville, TN 37240, USA}
\email{montsecasals@gmail.com}
\thanks{The first author is supported by Programa de Formaci\'{o}n
de Investigadores del Departamento de Educaci\'{o}n, Universidades e Investigaci\'{o}n del Gobierno Vasco}
\author[I. Kazachkov]{Ilya Kazachkov}
\email{ilya.kazachkov@gmail.com}
\thanks{The second author is supported by NSERC Postdoctoral Fellowship.}


\begin{document}

\begin{abstract}
In this note we characterise all finitely generated groups elementarily equivalent to a solvable Baumslag-Solitar group $\BS(m,1)$. It turns out that a finitely generated group $G$ is elementarily equivalent to $\BS(m,1)$ if and only if $G$ is isomorphic to $\BS(m,1)$.

Furthermore, we show that two Baumslag-Solitar groups are elementarily equivalent if and only if they are isomorphic.
\end{abstract}
\maketitle

\section{Introduction}

Classification of objects up to an equivalence relation is one the main problems in mathematics. The equivalence relation is determined by the type of property one would like to be preserved, namely in model theory one classifies modulo elementary theory; in algebra, modulo isomorphism; in geometric group theory, modulo quasi-isometries, etc. 
In this note we are concerned with the model-theoretic viewpoint, i.e. classification modulo elementary equivalence.

To determine which groups are elementarily equivalent to a given one is usually very hard. The only such example (aside from finite groups) is the classification of groups elementarily equivalent to an abelian group, this is a well-known result of  W.~Szmielew, \cite{Sz}.

When we consider only finitely generated groups, the problem is yet very hard. Here, the only known examples are the classification of finitely generated nilpotent groups and the recent solution of Tarski's problems. A theorem of  F.~Oger, \cite{Oger}, states that two finitely generated nilpotent groups $G$ and $H$ are elementarily equivalent if and only if $G\times \BZ\simeq H\times \BZ$.  In his paper \cite{Bel94}, O. Belegradek characterised groups that are elementarily equivalent to the group of unitriangular matrices $UT_n(R)$, where $R$ is an associative unital ring. In paper \cite{M33}, A. Miasnikov gave a description of groups elementarily equivalent to a given finitely generated nilpotent $K$-group, where $K$ is a field of characteristic zero. Recently, A. Miasnikov and M. Sohrabi, classified groups elementarily equivalent to a free nilpotent $R$-group of finite rank, where $R$ is a domain of characteristic zero, see \cite{MS}.

In their work O.~Kharlampovich and A.~Miasnikov, \cite{KhMTar}, and Z.~Sela, \cite{SelaTar} classify finitely generated groups elementarily equivalent to the free group. Sela's methods successfully generalise to all torsion-free hyperbolic groups, \cite{SelaHyp}.

In this note, we give the first example of a classification of finitely generated groups elementarily equivalent to some solvable groups, namely solvable Baumslag-Solitar groups. We prove
\begin{thm} \label{thm}
Let $\BS(1,n)$ be a Baumslag-Solitar group, $n\in \BZ$. Let $G$ be a finitely generated group elementarily equivalent to $\BS(1,n)$. Then $G$ is isomorphic to $\BS(1,n)$.
\end{thm}

If one considers the problem of classification of groups from a more specific class, more examples are known. 
\begin{itemize}
\item A classical result of A.~Malcev, \cite{Mal}, states that two classic linear groups ($SL$, $PSL$, $GL$, $PGL$) over a field are elementarily equivalent if and only if they are isomorphic. Malcev's results was generalised to other linear, algebraic and Chevalley groups by A.~Mikhal\"{e}v and E.~Bunina.
\item In \cite{CKR}, the authors showed that two right-angled Coxeter groups are elementarily equivalent if and only if they are isomorphic.
\item For solvable groups the only two results known to the authors are: a theorem of P.~Rogers, H.~Smith and D.~Solitar stating that finitely generated free solvable groups are elementarily equivalent if and only if they are isomorphic, see \cite{RSS}; and a theorem of Ch.~Gupta and E.~Timoshenko stating that two partially commutative metabelian groups are elementarily equivalent if and only if they are isomorphic, see \cite{GT}. 
\end{itemize}

The second result of this note is the classification of Baumslag-Solitar groups modulo elementary theories. We prove
\begin{thm} \label{thm0}
Two Baumslag-Solitar  groups $\BS(m,n)$ and $\BS(k,l)$ are elementarily (in fact $\forall\exists$-) equivalent if and only if there exists $\epsilon =\pm 1$, such that $m=\epsilon k$, $n=\epsilon l$ or $m=\epsilon l$, $n=\epsilon k$.
\end{thm}

As a consequence, we recover the main result from \cite{Mold} and obtain the following
\begin{cor}
For two Baumslag-Solitar groups $\BS(m,n)$ and $\BS(k,l)$ the following conditions are equivalent:
\begin{enumerate}
\item  either $m=\epsilon k$ and $n=\epsilon l$, or $m=\epsilon l$ and $n=\epsilon k$, where $\epsilon=\pm 1$;
\item  $\BS(m,n)$ and $\BS(k,l)$ are isomorphic;
\item  $\BS(m,n)$ and $\BS(k,l)$ are elementarily equivalent;
\end{enumerate}
\end{cor}

Recall that a Baumslag-Solitar group is a one-relator group given by the following presentation
$$
\BS(m,n)=\langle a, b\mid a^{-1}b^m a =b^n\rangle,
$$
where $m,n\in\BZ$, $m,n \ne 0$. Baumslag-Solitar groups were introduced in \cite{BS} as simple examples of non-Hopfian groups and since then they served as a ground for new ideas in group theory and as a testbed for theories and techniques. They were shown to have many remarkable properties. For instance, they are the simplest groups to have an exponential isoperimetric function, they were the first known groups to be asynchronously automatic but not automatic etc.

We assume that the reader is familiar with basic definitions and results from model theory and logic. We refer the reader to \cite{hodges} for notions not defined here.

We would like to thank the referee for his attention and care when reading our manuscript. We also thank D. Osin, who has pointed out to us that a result analogous to Theorem \ref{thm} was proved earlier by A. Nies, see \cite{Nies}. We are grateful to Gilbert Levitt, who pointed out a mistake in a previous version of the paper. 

\section{Proof of Theorems \ref{thm0} and \ref{thm}}  
It is clear that $\BS(m,n)\simeq \BS(n,m)$ and $\BS(m,n)\simeq \BS(-m,-n)$.  If $n=1$ and $m=1$, then $\BS(1,1)$ is abelian, $\BS(1,1)\simeq \BZ^2$. It is well-known that a finitely generated group $G$ is universally equivalent to $\BZ^2$ if and only if $G$ is a free abelian group of finite rank. Furthermore, by  Szmielew's theorem, a finitely generated group $G$ is elementarily equivalent to $\BS(1,1)$ if and only if $G\simeq \BS(1,1)\simeq \BZ^2$, see \cite{Sz}. Hence, without loss of generality, we further consider only the groups $\BS(m,n)=\langle a, b\mid a^{-1}b^m a=b^n\rangle$,  where $|m|\ge |n|$, $n>0$ and $mn\ne 1$.

\subsection{Proof of Theorem \ref{thm0}}
We define the first order formula $\Upsilon_{m,n} (x)$ in one free variable as follows:
$$
\Upsilon_{m,n} (x):
\begin{cases} 
\exists y \ (y^{-1} x^m y=x^n), & \text{ if $m\ne n$,}\\
\exists y \ ([y,x]\ne 1) \wedge (y^{-1} x^n y=x^n), &\text{ if $m=n$.}
\end{cases}
$$

The following lemma states two well-known facts about Baumslag-Solitar groups. Its proof is an application of the theory of HNN-extensions (see \cite[Propositions 1 and 2]{Mold} and \cite[Theorem 1]{An}).
\begin{lem}\label{lem00}\ 
\begin{enumerate}
\item  If $\BS(m,n)\models \Upsilon_{k,l}(g)$, then $g$ is conjugate to the element $b^p$ for some $p\in\BZ$.
\item  Let $m = m_1 d$, $n = n_1 d$, where $d = (m, n)$ is the greatest common divisor of $m$ and $n$, and let $p$ and $q$ be arbitrary different integers. Then the elements $b^p$ and $b^q$ are conjugate in $\BS(m, n)$ if and only if there exist integers $i$ and $r$, where $i> 0$, so that either $p=m_1^i d r$ and $q=n_1^i d r$ or $p=n_1^i d r$ and $q=m_1^i d r$.
\end{enumerate}
\end{lem}

We define the first order existential formula $\Theta_{m,n} (x)$ in one free variable as follows:
$$
\Theta_{m,n} (x):
\begin{cases} 
\exists y \ (y^{-1} x^m y=x^n) \wedge \bigwedge\limits_{d|(m,n)} y^{-1}x^{\frac{m}{d}}y\ne x^{\frac{n}{d}}, & \text{ if $m\ne n$,}\\
\exists y \ (y^{-1} x^n y=x^n) \wedge \bigwedge\limits_{d|n, d>1} y^{-1}x^{\frac{m}{d}}y\ne x^{\frac{n}{d}}, &\text{ if $m=n$.}
\end{cases}
$$

Clearly, $\BS(m,n)\models \exists x \Theta_{m,n}(x)$, as it suffices to take the generators $a$ and $b$ of $\BS(m,n)$ as  witnesses. Consider the group $\BS(k,l)$ and suppose that $\BS(k,l)\models \exists x \Theta_{m,n}(x)$. By Lemma \ref{lem00}(1), if $\BS(k,l)\models \Upsilon_{m,n}(x)$, then $x$ is a conjugate of $b^p$ for some $p\in \BZ$. Without loss of generality, we shall assume that $x=b^p$ for some (non-zero)  $p\in \BZ$. By Lemma \ref{lem00}(2), the elements $b^{pm}$ and $b^{pn}$ are conjugate in $\BS(k,l)$ if and only if $pm=k_1^i (k,l) r$ and $pn=l_1^i (k,l) r$, where $i \in \BN$, $0<r\in \BZ$, $k=k_1(k,l)$ and $l=l_1(k,l)$. Since $(k_1,l_1)=1$, it follows that $p$ divides $(k,l) r$. Therefore, $k_1$ divides $m$ and $l_1$ divides $n$. An analogous argument shows that if $\BS(m,n)\models \exists x \Theta_{k,l}(x)$, then $m_1$ divides $k$ and $n_1$ divides $l$.

Consider the sentence 
$$
\Phi_{m,n}: \exists x \forall z \bigwedge\limits_{1\neq d|(m,n)}  x\neq z^d \wedge \Theta_{m,n}(x).
$$
By the above $\BS(m,n)\models \Phi_{m,n}$ and if $\BS(k,l) \equiv \BS(m,n)$, then $|m|=|k|$ and $n=l$. We are left observe that, by Lemma \ref{lem00}, $\BS(m,n)\models \exists x \Upsilon_{m,n}(x)$ and $\BS(m,n)\not\models \exists x \Upsilon_{-m,n}(x)$. Hence, $\BS(m,n)$ and $\BS(-m,n)$ are not universally (existentially) equivalent. This finishes the proof of Theorem \ref{thm0}.

\subsection{Proof of Theorem \ref{thm}}
Recall that, without loss of generality, we consider only those Baumslag-Solitar groups $\BS(m,n)$ for which $|m|\ge n>0$. A Baumslag-Solitar group $\BS(m,n)$ is solvable if and only if it is metabelian if and only if $n=1$. 

For a group $\BS(m,1)$ there is an obvious homomorphism onto the infinite cyclic group, obtained by setting $b=1$. Standard techniques show that the kernel of this homomorphism is the commutator subgroup $\BS(m,1)'$ of $\BS(m,1)$, is generated by the elements of the form $a^{-i} b a^i$, where $i\in \BN$ and is isomorphic to the additive group of $m$-adic rational numbers $\BZ[\frac{1}{m}]$, see \cite{Col}.  Furthermore, for any $x\in \BS(m,1)$ the centraliser 
$$
C(x)= 
\begin{cases} 
\BS(m,1)', & \text{if $x\in\BS(m,1)'$,}
\\
\text{cyclic, } &\text{ otherwise.}
\end{cases}
$$

In the following lemma, we recall two well-known facts from model theory, see \cite{hodges}. 
\begin{lem} \label{lem0} \ 
\begin{enumerate}
\item 
Let $G$ be a group and let $N$ be a normal subgroup of $G$. Suppose that $N$ is interpretable in $G$. Then $\factor{G}{N}$ is interpretable in $G$.
\item 
Let $G_1$ and $G_2$ be elementarily equivalent groups. Let $H_1$ be interpretable in $G_1$ and $H_2$ be interpretable in $G_2$ using the same formulas, then $H_1$ is elementarily equivalent to $H_2$.
\end{enumerate}
\end{lem}

\begin{lem} \label{lem1} \
\begin{enumerate} 
\item Let $G$ be a group. There exists a first-order formula $\Psi_m(x)$ in the language of groups such that the truth set of $\Psi_m(x)$ is the set of all elements $x$ of $G$ such that the centraliser of $x$  is a normal, $m$-divisible, abelian subgroup containing the commutator subgroup $G'$.
\item The truth set of $\Psi_m(x)$ in $\BS(m,1)$ is $C(b)=\BS(m,1)'$.
\end{enumerate}
\end{lem}
\begin{proof}
We construct the formula $\Psi_m(x)$:
\begin{itemize}
\item $\forall y \ ([y,x]=1 \to \exists z \ y=z^m)$ - the centraliser of $C(x)$ is $m$-divisible;
\item $\forall y, z \  [x,[y,z]]=1$ - the centraliser of $x$ contains the commutator subgroup;
\item $\forall y, z \  [y,x]=1\to [y^z, x]=1$ - the centraliser of $x$ is a normal subgroup;
\item $\forall y, z \  (([y,x]=1 \wedge [z,x]=1) \to [z,y]=1)$ - the centraliser of $x$ is abelian.
\end{itemize}

From the description of centralisers in $\BS(m,1)$, it follows that the truth set of $\Psi_m(x)$ in $\BS(m,1)$ is the commutator subgroup $\BS(m,1)'=C(b)$.
\end{proof}

\begin{cor}
Let $\BS(m,1)$ be a non-abelian Baumslag-Solitar group. Then the subgroups $C(b)=\BS(m,1)'$ and $\factor{\BS(m,1)}{\BS(m,1)'}$ are interpretable in $\BS(m,1)$.
\end{cor}

Let now $G$ be a finitely generated group elementarily equivalent to $\BS(m,1)$. Since the properties ``to be metabelian'' and ``to be torsion-free'' are first-order, it follows that $G$ is metabelian and torsion-free. Furthermore, since $\BS(m,1)\models \forall x\forall y (\Psi_m(x) \wedge \Psi_m(y) \to [x,y]=1)$ and since $G$ is elementarily equivalent to $\BS(1,n)$, the truth set $A\subset G$ of the formula $\Psi_m(x)$ in $G$ is a normal, abelian, $n$-divisible subgroup containing the commutator subgroup $G'$ of $G$. By Lemma \ref{lem1}, the subgroup $A$ is elementarily equivalent to $\BS(m,1)'$ and the factor group $Q=\factor{G}{C(x)}$ is elementarily equivalent to $\BZ$. We therefore, obtain the following short exact sequence
\begin{equation} \label{eq1}
1\to A \hookrightarrow G \twoheadrightarrow Q \to 1.
\end{equation}

The group $G$ is finitely generated if and only if $Q$ is a finitely generated abelian group and $A$ is finitely generated as a $Q$-module. By Szmielew's theorem, \cite{Sz}, since $Q$ is finitely generated and elementarily equivalent to $\BZ$, it follows that $Q\simeq \BZ$. Similarly, we have that $\dim (\factor{A}{kA})=\dim(\factor{\BZ[\frac{1}{m}]}{k\BZ[\frac{1}{n}]})$ for all $k$. By the structure theorem for divisible abelian groups, \cite{Fuchs1}, since $A$ is finitely generated as a $Q$-module, it follows that $A\simeq \BZ[\frac{1}{m}]$.

Hence, to understand the structure of $G$, it now suffices to understand the action of $Q$ on $A$. There are two cases to consider:
\begin{itemize}
\item a generator of $Q$ acts on $A$ by multiplication by $m$, $-m$, $\frac{1}{m}$ or $-\frac{1}{m}$,
\item or $m=kl$ and $|k|\ne 1\ne |l|$, and a generator of $Q$ acts on $A$ by multiplication with $k/l$.
\end{itemize}

In the first case, the group $G\simeq \BS(m,1)$ or $G\simeq \BS(-m,1)$; in the second case, the group $G$ is given by the following presentation, 
$$
G=G_{k,l}=\langle a,b\mid a^{-1}b^k a=b^l, [b, a^{-i} b a^i]=1, \hbox{ where } i> 0\rangle,
$$
see \cite{Gild, Bieri, BaumStreb}.

By Theorem \ref{thm0}, it follows that $\BS(-m,1)$ is not elementarily equivalent to $\BS(m,1)$. We show that neither is  $G_{k,l}$. Consider the formula
$$
\Phi_m: \  \forall x \, \Psi_m(x)\,  \exists  y \in \factor{G}{\Psi_m(x)} \ x^y = x^m.
$$
By Lemma \ref{lem0}, $\Phi_m$ is a first-order sentence in the language of groups. Obviously,  $\BS(m,1)\models \Phi_m$. Direct computations show that, on one hand, $G_{k,l}\models \Phi_m(b)$ and, on the other, $b^y \ne b^m$ for all $y\in \factor{G_{k,l}}{G_{k,l}'}$. Hence $G_{k,l}\not\models \Phi_m$. This finishes the proof of Theorem \ref{thm}.

\end{document}